\documentclass[letterpaper]{article}

\newif\ifDRAFT

\usepackage[T1]{fontenc}
\usepackage[utf8]{inputenc}
\usepackage[english]{babel}
\usepackage[autostyle]{csquotes}
\usepackage[full]{textcomp}
\usepackage{amsfonts, amsmath, amsthm, amssymb, mathtools}
\usepackage{thm-restate}
\usepackage{tikz-cd}
\usepackage{upgreek}
\usepackage{cite}
\usepackage{hyperref}
\usepackage{cleveref}
\usepackage{xcolor}
\usepackage{fullpage}
\usepackage{booktabs}

\hypersetup{
    colorlinks,
    linkcolor={red!50!black},
    citecolor={blue!50!black},
    urlcolor={blue!80!black}
}

\usepackage{color}
\definecolor{keywordcolor}{rgb}{0.7, 0.1, 0.1}   
\definecolor{commentcolor}{rgb}{0.4, 0.4, 0.4}   
\definecolor{symbolcolor}{rgb}{0.0, 0.1, 0.6}    
\definecolor{sortcolor}{rgb}{0.1, 0.5, 0.1}      
\definecolor{errorcolor}{rgb}{1, 0, 0}           
\definecolor{stringcolor}{rgb}{0.5, 0.3, 0.2}    

\newcommand{\mathlib}{\texttt{mathlib}}
\newcommand{\lean}{\lstinline}

\usepackage{listings}

\numberwithin{equation}{section}
\theoremstyle{plain}
\newtheorem{theorem}{Theorem}[subsection]

\newtheorem{lemma}[theorem]{Lemma}

\theoremstyle{definition}
\newtheorem{definition}[theorem]{Definition}

\theoremstyle{remark}
\newtheorem{remark}[theorem]{Remark}

\ifDRAFT
\usepackage[inline]{showlabels}

\showlabels{cite}
\showlabels{bibitem}

\usepackage[yyyymmdd,hhmmss]{datetime}
\usepackage{background}
\SetBgContents{DRAFT \today{} \currenttime{}}
\SetBgScale{1}
\SetBgAngle{0}
\SetBgColor{black}
\SetBgOpacity{0.5}
\SetBgPosition{current page.south}
\SetBgVshift{1.5cm}
\fi

\newcommand\blfootnote[1]{%
  \begingroup
  \renewcommand\thefootnote{}\footnote{#1}%
  \addtocounter{footnote}{-1}%
  \endgroup
}

\title{Formalizing Hall's Marriage Theorem in Lean}
\author{Alena Gusakov\footnote{Department of Computer and Information Sciences, University of Delaware, Newark, Delaware 19716, USA}, Bhavik Mehta\footnote{Department of Pure Mathematics and Mathematical Sciences, University of Cambridge, UK}, Kyle A. Miller\footnote{Department of Mathematics, University of California, Berkeley, California 94720-3840, USA}}
\date{December 2020}

\begin{document}

\maketitle

\blfootnote{Contact: \texttt{agusakov@udel.edu}, \texttt{bm489@cam.ac.uk}, and \texttt{kmill@math.berkeley.edu}}

%

\begin{abstract}
    We formalize Hall's Marriage Theorem in the Lean theorem prover for inclusion in \mathlib{}, which is a community-driven effort to build a unified mathematics library for Lean.
    One goal of the \mathlib{} project is to contain all of the topics of a complete undergraduate mathematics education.
    
    We provide three presentations of the main theorem statement: in terms of indexed families of finite sets, of relations on types, and of matchings in bipartite graphs.
    We also formalize a version of K\H{o}nig's lemma (in terms of inverse limits) to boost the theorem to the case of countably infinite index sets.
    We give a description of the design of the recent  \mathlib{} library for simple graphs, and we also give a necessary and sufficient condition for a simple graph to carry a function.
    
    
    \vspace{1em}
    
    \noindent\emph{Keywords:} Hall's Marriage Theorem, computer formalization, graph theory.
    
    \vspace{0.4em}
    
    \noindent Mathematics Subject Classification 2020: 05-04, 05C70; 68R05
\end{abstract}

\section{Introduction}

Mathematics communities surrounding proof assistants have long been interested in their applications to research\cite{Buzzard2019}\cite{Hales2017}. However, most existing formalizations of results such as these tend to be self-contained and cannot be found in any sort of larger general-purpose libraries\cite{Singh2017} or only contain a partial formalization of their main result\cite{Strickland2019}. For this reason, mathematics communities surrounding proof assistants are currently working to formalize the standard undergraduate-level mathematics curriculum, with the goal of building up to graduate- and research-level results.

Combinatorics and graph theory, part of the standard undergraduate discrete mathematics curriculum, are somewhat hard to find within existing proof assistant libraries.
An early formalization of graphs was \cite{Chou1994}, which in HOL defined multigraphs and proved basic facts about paths, reachability, and trees.
Planar graph theory formalization came in 2004, when Georges Gonthier and Benjamin Werner formalized the proof of the Four Color Theorem by Neil Robertson et al. in Coq \cite{Gonth2008}, followed by formalizations of theorems including Dilworth's Theorem in 2017 by Abishek Singh\cite{Singh2017}, and the weak perfect graph theorem by A. K. Singh and R. Natarajan in 2019\cite{Singh2019}. In 2020, Christian Doczkal and Damien Pous formalized Menger's Theorem as well as some results involving minors and treewidths of graphs, and proposed a general-purpose graph theory library for Coq \cite{Doczkal2019}. However, none of these results fall under any sort of graph theory library that is integrated with other areas of math within the respective proof assistant. 

By contrast, the community surrounding the Lean proof assistant is working to create a unified library of mathematical results across all fields called \mathlib{}\cite{mathlib2019}. First introduced in 2013\cite{Moura2018}, Lean has not been around for as long as other more well-established proof assistants. However, there are several technical and logistical features that make Lean a good tool for creating such a library. Its implementation using dependent type theory, based on the Calculus of Inductive Constructions (CIC), makes automation and verification of Lean easier than proof assistants based on ZFC, and it makes things like algebraic reasoning easier in Lean than proof assistants based on simple type theory. The most comparable proof assistant to Lean is Coq, which has its own mathematics library called the Mathematical Components library. Many community members and contributors surrounding both Lean and Coq have mathematics backgrounds.  There are subtle differences in the philosophies held by contributors: libraries using Coq tend to be self-contained without any dependence on substructures, and they also tend to be exclusively constructive, whereas mathematical objects in \mathlib{} inherit properties of underlying structures, and classical logic is used where needed \cite{mathlib2019}.
 
Until Summer of 2020, \mathlib{} did not have many combinatorial or graph theoretical concepts, but since then there has been a small but steadily growing library for these topics to which the authors of this paper have contributed. One notable undergraduate-level theorem missing from the graph theory and combinatorics libraries is Hall's Marriage Theorem. Part of the standard undergraduate discrete mathematics curriculum, Hall's Marriage Theorem has multiple equivalent formulations within combinatorics, from matchings in graphs to distinct representatives of sets. 

In this paper, we document the formalization of Hall's Marriage Theorem in Lean. This is not the first time it has been formalized in a proof assistant: in \cite{Singh2017}, the proof of Dilworth's Theorem is specialized to a proof of Hall's Marriage Theorem in Coq. However, the proof is limited to the context of sets and sequences. We propose formulations of three different versions of Hall's Marriage Theorem: via indexed sets, relations on types, and matchings in bipartite graphs.
The formalized proof of these are in the process of being incorporated into \mathlib{}.\footnote{The code resides in \url{https://github.com/leanprover-community/mathlib/tree/simple_graph_matching} before it is merged into \mathlib{} proper.}

The organization of this paper is as follows.
\Cref{sec:background} recalls Hall's Marriage Theorem and formulates it in three different ways, where \Cref{sec:hall-proof} gives the specific version of the proof that we formalize.
\Cref{sec:formalization} gives formalizations of each of these formulations, and for the relation-based formulation we give an overview of its Lean proof in \Cref{sec:hall-theorem-formal}.
Included in this section is \Cref{sec:konig-lemma-formal}, which has a description of a formulation of K\H{o}nig's lemma along with its application to proving a formalization of Hall's Marriage Theorem with a countably infinite index set.
We also give a graph-theoretic application in \Cref{sec:application}, giving a necessary and sufficient condition for a simple graph to carry a function.
We briefly discuss future work in \Cref{sec:future}.

\section{Background}
\label{sec:background}

For a set $S$, consider a family of finite subsets $X_1,\dots,X_n\subseteq S$.
In \cite{Hall1935}, Philip Hall gave a necessary and sufficient condition for there to exist a set of distinct representatives $x_1,\dots,x_n\in S$ with $x_i\in X_i$ for each $1\leq i \leq n$. The condition is that
\begin{align}
    \label{eq:marriage-condition}
    \text{for every subset $J\subseteq \{1,\dots,n\}$ that }\lvert J\rvert \leq \left\lvert \bigcup_{i\in J} X_i\right\rvert.
\end{align}
This result is known as \emph{Hall's Marriage Theorem} due to the story with which it is usually presented.
Instead, we will give an alternative story that we believe is more intuitive.
A group of people go to a dog shelter, and they each have a (finite) list of dogs they would prefer to adopt.
The question is what are the precise conditions under which everyone will be matched with a dog they prefer.
From this point of view, \emph{Hall's marriage condition} (\Cref{eq:marriage-condition}) is that, for every subset of $k$ people, the length of their combined list of preferred dogs is at least $k$.

In the following sections, we give three formulations of Hall's Theorem in terms of indexed families of sets, of relations on sets, and of matchings in bipartite graphs.
Each of these formulations will be given Lean formalizations in \Cref{sec:formalization}.


\subsection{Via indexed families of sets}
\label{sec:hall-via-indexed}

We will restate Hall's Marriage Theorem in terms of indexed families of finite subsets.
\begin{definition}
    For a fixed set $S$, a \emph{family of finite subsets $\{X_i\}_{i\in I}$ indexed by a set $I$} is a collection of subsets $X_i\subseteq S$ for each $i\in I$.
    The set $I$ is called the \emph{index set}.
    An element $x\in \prod_{i\in I}X_i$ is called a \emph{family of elements} of the indexed family, and it may be regarded as a function $I\to S$ with $x_i:=x(i)$ with $x_i\in X_i$ for each $i\in I$.
\end{definition}

\begin{definition}
    A \emph{matching} (or \emph{transversal}) of an indexed family of subsets $\{x_i\}_{i\in I}$ is a family of elements $x$ that is injective when thought of as a function $I\to S$, which is to say that $x_i=x_j$ implies $i=j$.
\end{definition}

\begin{theorem}[{Hall's Marriage Theorem\cite{Hall1935}}]
    \label{thm:hall-indexed-family}
    Let $\{X_i\}_{i\in I}$ be an indexed family of finite subsets with finite index set $I$.
    The indexed family has a matching if and only if for all $J\subseteq I$, we have $\lvert J\rvert \leq \left\lvert \bigcup_{i\in J} X_i\right\rvert$.
\end{theorem}

We will give a proof of a variant of this theorem in \Cref{sec:hall-proof}.
Hall's Marriage Theorem was extended to infinite families by Marshall Hall:
\begin{theorem}[{Infinite Hall's Marriage Theorem\cite{Hall1986}}]
    \label{thm:infinite-hall-indexed-family}
    Let $\{X_i\}_{i\in I}$ be an indexed family of finite subsets with an index set $I$ of any cardinality.
    The indexed family has a matching if and only if for all finite subsets $J\subseteq I$ then $\lvert J\rvert \leq \left\lvert \bigcup_{i\in J} X_i\right\rvert$.
\end{theorem}
\begin{proof}[Proof sketch]
    The hard direction, that the marriage condition implies the existence of a matching, follows from the Compactness Theorem from logic, but there is also a topological proof in \cite{Halpern1966}.
    For each finite subset $J\subseteq I$, let $M_J$ denote the set of all matchings on the indexed family $\{X_i\}_{i\in J}$.
    This indexed family satisfies the marriage condition, hence $M_J$ is nonempty for each $J$ by \Cref{thm:hall-indexed-family}.
    If $J\subseteq J'$, then there is a map $M_J'\to M_J$ given by restriction of the matching, and these maps define an inverse system.
    Since each $M_J$ is finite, and thus compact, the inverse limit $\varprojlim_{J}M_J$ is nonempty since it is an intersection of closed subsets with the finite intersection property of $\prod_J M_J$, which is compact by the Tychonoff theorem.
    Any element of $\varprojlim_JM_J$ gives a matching for the original indexed family.
\end{proof}

In the case that $I$ is countably infinite, then the above argument follows from a variant of  K\H{o}nig's lemma, which we will describe in more detail in \Cref{sec:konig-lemma-formal}.

\subsection{Via relations}
\label{sec:hall-via-relations}

For sets $A$ and $B$, consider a relation $r$ between $A$ and $B$, with $r\ a\ b$ indicating that $a\in A$ is related to $b\in B$ by $r$.
For a subset $S\subseteq A$, let $r(S)$ denote the set $\{b\in B\mid \exists a\in A, r\ a\ b\}$.

\begin{definition}
    Given a relation $r$ between sets $A$ and $B$, a \emph{matching of $r$ that saturates a subset $S\subseteq A$} is an injective function $f:S\to B$ that \emph{respects} the relation $r$, which is to say that $r\ a\ f(a)$ for all $a\in S$.
    A matching that saturates $A$ is simply called a matching.
\end{definition}

\begin{restatable}[Hall's Marriage Theorem]{theorem}{thmHallRelations}
    \label{thm:hall-relation}
    Let $r$ be a relation between a finite set $A$ and a finite set $B$.
    The relation has a matching that saturates $A$ if and only if for all $S\subseteq A$ then $\lvert S\rvert \leq \lvert r(S)\rvert$.
\end{restatable}
\begin{proof}
    Consider $A$ to be the index set for the family $\{X_a\}_{a\in A}$ with $X_a:=\{b\in B\mid r\ a\ b\}$.
    Then this theorem follows from \Cref{thm:hall-indexed-family}.
\end{proof}

This theorem is equivalent to \Cref{thm:hall-indexed-family}, since $S$ in that theorem can be restricted to $\bigcup_{i\in I}X_i$, which is a finite set.
One can also restate \Cref{thm:hall-relation} so that, instead of requiring that $B$ be finite,  each $a\in A$ is related to only finitely many elements of $B$.

\subsection{Via graph theory}
\label{sec:hall-via-graphs}

A \emph{(simple) graph} $G$ on a set $V$ of \emph{vertices} is a symmetric irreflexive binary relation on $V$, where vertices $v,w\in V$ are \emph{adjacent} if they are related by this relation.
An \emph{edge} of $G$ is an unordered pair of adjacent vertices, and the set of all edges of $G$ is denoted $E(G)$; the vertices comprising an edge are said to be \emph{incident} to it.
For subsets $S\subseteq V$ of vertices, the \emph{neighborhood} $\Gamma(S)$ of $S$ is the set of all vertices in $V$ adjacent to at least one vertex in $S$.

\begin{definition}
    A \emph{matching} $M$ on a graph $G$ is a subset $M\subseteq E(G)$ of edges such that distinct edges of $M$ share no incident vertices.
    The matching is said to \emph{saturate} a subset $W\subseteq V$ if every vertex of $W$ is incident to an edge of $M$.
\end{definition}

\begin{definition}
    A \emph{(proper) coloring} of a graph $G$ with color set $C$ is a function $f:V\to C$ assigning colors to each vertex such that adjacent vertices have different colors.
    For color $c\in C$, the \emph{color class} associated to $c$ is $f^{-1}(c)$.
\end{definition}

\begin{definition}
    A \emph{bipartition} of a graph $G$ is a coloring of $G$ with color set $\{1,2\}$.
    Let $V_1$ and $V_2$ respectively denote the color classes for colors $1$ and $2$.
    If a bipartition exists, the graph is called \emph{bipartite}.
\end{definition}

\begin{theorem}[Hall's Marriage Theorem]
    \label{thm:hall-graph}
    Let $G$ be a bipartitioned simple graph with $V_1$ finite and $\Gamma(v)$ finite for each $v\in V_1$.
    $G$ has a matching that saturates $V_1$ if and only if for all $S\subseteq V_1$ then $\lvert S\rvert \leq \lvert \Gamma(S)\rvert$.
\end{theorem}
\begin{proof}
    A bipartioned simple graph is precisely a relation between $V_1$ and $V_2$, hence \Cref{thm:hall-relation} is equivalent.
\end{proof}

\subsection{Proof of Hall's Marriage Theorem}
\label{sec:hall-proof}


For reference, we give a proof of Hall's theorem, mirroring the treatment in \cite{Rosen2011}.
We give it in terms of relations as in \Cref{thm:hall-relation}, which we recall:
\thmHallRelations*
\begin{proof}
    First suppose that there exists a matching $M$ that saturates $A$.
    If $S \subseteq A$, then since $M$ saturates $A$ it must also saturate $S$. 
    If $M(S)$ 
    denotes the image of $S$ by $M$ in $B$, then $|S| = |M(S)|$ by injectivity. 
    Since $M(S) \subseteq r(S)$, we have that $|S| = |M(S)| \leq |r(S)|$.
    
    The converse is the ``hard'' direction.
    We proceed by strong induction on $n=\lvert A\rvert$.
    \begin{description}
    \item[Base case ($n=0$):] This means that $A = \emptyset$. The empty matching saturates $\emptyset$.
    
    \item[Base case ($n=1$):] This means that $A = \{a\}$ for some $a$, hence every $S \subseteq A$ is either the empty set or $\{a\}$. Since we have that $|S| \leq |r(S)|$ for every $S \subseteq A$, we know that $|\{a\}| \leq |r(\{a\})|$, so there exists some $b \in B$ such that $r\ a\ b$. We can define our matching as the function $f:A\to B$ such that $f(a) = b$.
    
    \item[Induction hypothesis:] If $r$ is a relation between a finite set $A$ with $|A|\leq k$ and a finite set $B$, then if $|S| \leq |r(S)|$ for every $S \subseteq A$, there exists a matching of $r$ that saturates $A$.
    
    \item[Induction step:] Suppose $|A| = k + 1$ and $|S| \leq |r(S)|$ for every $S \subseteq A$.
    We have two cases: either (1) every proper nonempty subset $S\subsetneq A$ satisfies $|S| < |r(S)|$ or (2) there is some proper nonempty subset $S\subsetneq A$ such that $|S| = |r(S)|$.
        \begin{description}
        \item[Case 1:] Assume for every nonempty subset $S \subsetneq A$ that $|S| < |r(S)|$, and choose arbitrary $a \in A$ and $b \in r(\{a\})$.
        Set $A':=A \setminus \{a\}$ and $B' := B\setminus \{b\}$, and let $r'$ be the restriction of $r$ to $A'$ and $B'$.
        We prove that Hall's condition is satisfied for $r'$.
        Let $T \subseteq A'$. Since $|T| < |r(T)|$, we know that $|T| + 1 \leq |r(T)|$, and removing $b$ from $B$ gives us $|r(T)| - 1 \leq |r'(T)|$, so we now have that $|T| \leq |r'(T)|$. 
        By our induction hypothesis, there exists a matching $M':A'\to B'$, which can be extended to a matching $M:A\to B$ with $M(a)=b$.

        \item[Case 2:] There exists some proper nonempty $S_0 \subsetneq A$ such that $|S_0| = |r(S_0)|$. 
            We first prove that Hall's condition is satisfied for $S_0$. 
            We restrict $r$ to a relation $r'$ between $S_0$ and $r(S_0)$, hence for $T \subseteq S_0$ we have $r(T) = r'(T)$.
            Since for all $T \subseteq S_0$, $|S_0| \leq k$ and $|T| = |r'(T)|$, by our induction hypothesis there is a matching $M_0$ of $r'$ that saturates $S_0$.
            
            Now we consider $A'' = A \setminus S_0$ and $B'' = B \setminus r(S_0)$.
            Let $r''$ be the restriction of $r$ to $A''$ and $B''$.
            Thus, for $T \subseteq A'$,
            \begin{equation*}
                r''(T) = \{y\ |\ r\ x\ y \text{ for some }x\in T\text{ and }y \in B'\}.
            \end{equation*}
            Since $T$ and $S_0$ are disjoint and $r''(T)$ and $r'(S_0)$ are disjoint, we have that $|S_0 \cup T| = |S_0| + |T|$, and $r(S_0 \cup T) = r'(S_0) \cup r''(T)$ so therefore $|r(S_0 \cup T)| = |r'(S_0)| + |r''(T)|$. Since $|S| \leq |r(S)|$ for all $S \subseteq A$, we have that $|S_0| + |T| = |S_0 \cup T| \leq |r(S_0 \cup T)| = |r'(S_0)| + |r''(T)|$, so $|S_0| + |T| \leq |r'(S_0)| + |r''(T)|$. Since $|S_0| = |r'(S_0)|$, we therefore have $|T| \leq |r''(T)|$ for all $T \subset A''$. By our induction hypothesis, this means we have a matching $M_1$ for $r''$ that saturates $A''$.

            Since the domains of $M_0$ and $M_1$ are disjoint, we can define a matching $M$ that saturates $A$ by $M(a)=M_0(a)$ for $a\in S_0$ and $M(a)=M_1(a)$ otherwise.
        \end{description}
    \end{description}
    This completes the proof.
\end{proof}

\section{Formalization}
\label{sec:formalization}


We present formalizations of the statements of each of the three formulations of Hall's Marriage Theorem in \Cref{sec:hall-via-indexed,sec:hall-via-relations,sec:hall-via-graphs}.
For the relation-based presentation in \Cref{sec:formal-via-relations}, we go through the formalized proof in more detail, mirroring the proof in \Cref{sec:hall-proof}.
We use the relation-based formalization to prove the graph-based one using the \mathlib{} simple graph library.

For clarity, we omit all references to decidability, such as the \lean{decidable}, \lean{decidable_eq}, and \lean{decidable_pred} typeclasses.
Like Coq, Lean is based on the calculus of inductive constructions, but it has a noncumulative hierarchy of universes \lean{Prop}, \lean{Type 0}, \lean{Type 1}, and so on, where \lean{Prop} is an impredicative universe of propositions such that (1) proofs of the same propositions are equal and (2) logically equivalent propositions are equal.
While terms of \lean{Type 0} and above imply some kind of computational content, Lean does not promise to be able to compute whether a proposition in \lean{Prop} is \lean{True} or \lean{False}, but instead it provides an extensive facility to decide the truth of certain propositions using the typeclass interface.
This is, in particular, used pervasively in \mathlib{}'s library for finite sets, which are defined to be the quotient type of finite lists up to permutation.

We also omit the tactic proofs and instead opt to describe their contents, since without a proof assistant on hand the descriptions are more illuminating.


\subsection{Via indexed families of sets}
\label{sec:formalize-indexed-families}

In Lean, we can represent an indexed family of finite sets as a function \lean{ι : α → finset β} where \lean{α} is the index set.
A matching then is given by an injective function \lean{f : α → β} such that \lean{∀ (a : α), f a ∈ ι a}.
Hence, consider the following variables to introduce such an indexed family.
\begin{lstlisting}
universes u v
variables {α : Type u} {β : Type v} (ι : α → finset β)
\end{lstlisting}
Lean has some support for universe polymorphism.
The universe variables \lean{u} and \lean{v} stand in for the natural numbers.

We decided to introduce a ``bundled'' type for a matching, which consists of both the function along with the proofs that it is an element of the indexed family of sets.
\begin{lstlisting}
structure matching :=
(f : α → β)
(mem_prod' : ∀ (a : α), f a ∈ ι a)
(injective' : injective f)
\end{lstlisting}
The type of all matchings for an indexed family \lean{ι} is given by \lean{matching ι}.
Lean has a coercion facility to automatically turn terms of types into types or functions, and we implemented function coercion for \lean{m : matching ι} so that \lean{m a} equals \lean{m.f a} for all \lean{a : α}.
Since Lean has proof irrelevance, \lean{m = m'} if and only if \lean{m.f = m'.f}, so a matching ``is'' its function.

With these basic definitions, we can state Hall's Marriage Theorem on a finite index type:
\begin{lstlisting}
theorem hall [fintype α] :
  (∀ (s : finset α), s.card ≤ (s.bind ι).card) ↔ nonempty (matching ι)
\end{lstlisting}
The expression \lstinline{s.bind ι} denotes the finite set $\bigcup_{x\in s}\iota(x)$, where the name \lstinline{bind} derives from the language of monads in functional programming.

We also formalized a version of K\H{o}nig's lemma, and with it we proved the countably infinite version of Hall's Marriage Theorem:
\begin{lstlisting}
theorem infinite_hall (h : ℕ ≃ α) :
  (∀ (s : finset α), s.card ≤ (s.bind ι).card) ↔ nonempty (matching ι)
\end{lstlisting}
We will discuss the details of this in \Cref{sec:konig-lemma-formal}.

\subsection{Via relations}
\label{sec:formal-via-relations}

The relation-based formulation statement of Hall's Marriage Theorem from \Cref{sec:hall-via-relations} and its proof in \Cref{sec:hall-proof} are roughly translated in the following way.

Let $r$ be a relation on finite types $\alpha$ and $\beta$.
If $A$ is a finite set of terms of $\alpha$, we denote the image $r(A)$ of $A$ in $\beta$ in Lean by \lstinline{image_rel} $r\ A$.
\begin{lstlisting}
variables {α β : Type u} [fintype α] [fintype β]
variables (r : α → β → Prop)
def image_rel (A : finset α) : finset β := univ.filter (λ b, ∃ a ∈ A, r a b)
\end{lstlisting}
The requirement that \lean{β} be a finite type makes it convenient in defining \lean{image_rel}.
We could instead have had a hypothesis that \lean$[∀ (a : α), fintype {b : β | r a b}]$ to require that each term \lean{a : α} be related to only finitely many terms of \lean{β}.

 Recall that Hall's Marriage Theorem is that the following two statements are equivalent:
 \begin{itemize}
     \item For every finite set $A$ with terms in $\alpha$, then $|A| \leq\ |r(A)|$.
     \item There exists an injective function $f : \alpha \rightarrow \beta$ such that $r\ x\ (f(x))$ for all $x$ in $\alpha$.
 \end{itemize}
 In Lean, this is represented as:
\begin{lstlisting}
theorem hall :
  (∀ (A : finset α), A.card ≤ (image_rel r A).card)
    ↔ (∃ (f : α → β), function.injective f ∧ ∀ x, r x (f x))
\end{lstlisting}

Our formalized proof of this theorem is broken into separate lemmas, mirroring the proof in \Cref{sec:hall-proof}.
We first consider the converse:
\begin{lstlisting}
theorem hall_easy (f : α → β) (hf₁ : function.injective f) (hf₂ : ∀ x, r x (f x)) 
(A : finset α) : A.card ≤ (image_rel r A).card
\end{lstlisting}
For the proof of \lstinline{hall_easy}, we simply show that the image of $f$ in $\beta$ is a subset of $r(A)$. Since $f$ is injective, the cardinality of its image is equal to that of its preimage, so $|A| = |f(A)| \leq |r(A)|$ and therefore $|A| \leq\ |r(A)|$.

For the forward direction, we use strong induction on the cardinality of $\alpha$.
Strong induction results in the induction hypothesis that
for all relationa $r'$ between finite types $\alpha'$ and $\beta'$ such that $|\alpha'| \leq k$, we have that if $|A| \leq\ |r'(A)|$ for all finite sets $A$ on $\alpha'$, then there exists an injective function $f' : \alpha' \rightarrow \beta'$ such that $r'\ x\ f(x)$ for all $x$ in $\alpha'$.
This is represented as
\begin{lstlisting}
(ih : ∀ {α' β' : Type u} [fintype α'] [fintype β']
        (r' : α' → β' → Prop),
        fintype.card α' ≤ n →
        (∀ (A : finset α'), A.card ≤ (image_rel r' A).card) →
        ∃ (f : α' → β'), function.injective f ∧ ∀ x, r' x (f x))
\end{lstlisting}
In Lean, the type \lean{set α} of sets is defined by \lean{α → Prop}, which are essentially indicator functions on \lean{α}.
Due to proof irrelevance, sets can be easily coerced to types, so it is easy to construct the required restrictions of \lean{r} and apply \lean{ih}.

The first two lemmas are for the two base cases for $|\alpha| = 0$ and $|\alpha| = 1$.

\begin{lstlisting}
theorem hall_hard_inductive_zero (hn : fintype.card α = 0)
  (hr : ∀ (A : finset α), A.card ≤ (image_rel r A).card) :
  ∃ (f : α → β), function.injective f ∧ ∀ x, r x (f x)
\end{lstlisting}
\begin{proof}
    This is straightforward to prove --- if $\alpha$ has cardinality $0$, this means it is empty, so any any finite sets defined on $\alpha$ will be the empty set. Finding an injective function from the empty set to $\beta$ is trivial.
\end{proof}

\begin{lstlisting}
theorem hall_hard_inductive_one (hn : fintype.card α = 1)
  (hr : ∀ (A : finset α), A.card ≤ (image_rel r A).card) :
  ∃ (f : α → β), function.injective f ∧ ∀ x, r x (f x)
\end{lstlisting}
\begin{proof}
    If $\alpha$ has cardinality $1$, the only finite sets we can define on it are the empty set and the singleton set. Letting $a$ be the single element of $\alpha$, we have that $|\{a\}| \leq\ |r(\{a\})|$. This means that there's at least one element in $r(a)$, so we can select any one of them to define our injective function.
\end{proof}

We now consider the induction step with $|\alpha| \leq k + 1$.
This is set up with the following lemma:
\begin{lstlisting}
theorem hall_hard_inductive_step [nontrivial α] {n : ℕ} (hn : fintype.card α ≤ n.succ)
  (hr : ∀ (A : finset α), A.card ≤ (image_rel r A).card)
  (ih : ∀ {α' β' : Type u} [fintype α'] [fintype β'] (r' : α' → β' → Prop),
    fintype.card α' ≤ n →
    (∀ (A' : finset α'), A'.card ≤ (image_rel r' A').card) →
    ∃ (f' : α' → β'), function.injective f' ∧ ∀ x, r' x (f' x)) :
  ∃ (f : α → β), function.injective f ∧ ∀ x, r x (f x)
\end{lstlisting}
\begin{proof}
    We consider two cases, whether or not every finite set $A$ on $\alpha$ that is neither the empty set nor all of $\alpha$ is such that $|A| < |r(A)|$.
    Each of these cases are separately handled by \lean{hall_hard_inductive_step_1} and \lean{hall_hard_inductive_step_2} below.
\end{proof}

In the first case, which is when this condition holds, the following lemma applies:
\begin{lstlisting}
lemma hall_hard_inductive_step_1 [nontrivial α] {n : ℕ}
  (hn : fintype.card α ≤ n.succ)
  (ha : ∀ (A : finset α), A.nonempty → A ≠ univ → A.card < (image_rel r A).card)
  (ih : ∀ {α' β' : Type u} [fintype α'] [fintype β'] (r' : α' → β' → Prop),
    fintype.card α' ≤ n →
    (∀ (A' : finset α'), A'.card ≤ (image_rel r' A').card) →
    ∃ (f' : α' → β'), function.injective f' ∧ ∀ x, r' x (f' x)) :
  ∃ (f : α → β), function.injective f ∧ ∀ x, r x (f x)
\end{lstlisting}
\begin{proof}
    Since we assumed that $A \neq \emptyset$ and $\alpha$ is nontrivial, we have that there are at least two elements $a, a' \in \alpha$ such that $a \neq a'$. Because $\{a\}$ is a proper nonempty subset of $\alpha$, we have that $|\{a\}| <\ |r(\{a\})|$. This means that we can select some $b \in r(A\setminus \{a\})$. Informally, we define $\alpha'$ as $\alpha \setminus \{a\}$, and similarly, we define $\beta'$ as $\beta \setminus \{b\}$. We also define $r'$ as $r$ restricted to $\alpha'$ and $\beta'$. Formally stated, within the proof we have
\begin{lstlisting}
let α' := {a' : α // a' ≠ a},
let β' := {b' : β // b' ≠ b},
let r' : α' → β' → Prop := λ a' b', r a' b',
\end{lstlisting}
    It's clear that for all finite sets $A'$ defined on $\alpha'$, $|A'| \leq\ |r'(A')|$. This is because a finite set $A'$ defined on $\alpha'$ will be equivalent to $A \setminus \{a\}$ for some finite set $A$ defined on $\alpha$. This gives us two cases: either $a \in A$ or $a \notin A$. 

    If $a \in A$, we get that $A' := A \setminus \{a\}$ in $\alpha'$ has cardinality $|A| - 1$. Furthermore, if $a \in A$ then $b \in r(A)$, so then the cardinality of $r'(A')$ is $|r(A)| - 1$. By our assumption, $|A| < |r'(A)|$, so this gives us that $|A'| < |r'(A)|$ in $\alpha'$.

    If $a \notin A$, we get that $A' := A \setminus \{a\}$ in $\alpha'$ has cardinality $|A|$. If $b$ is related to some other element of $A$ by $r$, then $|r'(A')|$ in $\beta'$ will have cardinality $|r(A)| - 1$ and thus $|A'| \leq r'(A')|$. Otherwise, if $b \notin r(A)$, we have that $|r'(A')|$ in $\beta'$ will have cardinality $|r(A)|$, so $|A'| < r'(A')|$.

    Then we can plug $\alpha'$, $\beta'$, and $r'$ into our induction hypothesis to get that there is some injective function $f' : \alpha' \rightarrow \beta'$ that respects $r'$. 

    To get our injective function $f : \alpha \rightarrow \beta$, we simply combine $f'$ with a mapping from $a$ to $b$. We know that it is injective because $\{b\}$ is disjoint from $\beta'$, and we know that it respects relation $r$ because $b \in r(\{a\})$ and $f'$ respects $r'$.
\end{proof}

In case the condition does not hold, there exists some nonempty proper finite set $A$ of $\alpha$ such that $|A| =\ |r(A)|$.
\begin{lstlisting}
lemma hall_hard_inductive_step_2 [nontrivial α] {n : ℕ} 
  (hn : fintype.card α ≤ n.succ)
  (hr : ∀ (A : finset α), A.card ≤ (image_rel r A).card)
  (ha : ∃ (A : finset α),  A.nonempty ∧ A ≠ univ ∧ A.card = (image_rel r A).card) 
  (ih : ∀ {α' β' : Type u} [fintype α'] [fintype β'] (r' : α' → β' → Prop),
    fintype.card α' ≤ n →
    (∀ (A' : finset α'), A'.card ≤ (image_rel r' A').card) →
    ∃ (f' : α' → β'), function.injective f' ∧ ∀ x, r' x (f' x)) : 
  ∃ (f : α → β), function.injective f ∧ ∀ x, r x (f x)
\end{lstlisting}
\begin{proof}
    Let $A$ be a finite set defined on $\alpha$ such that $|A| =\ |r(A)|$. We define two subtypes of $\alpha$, leading us to two injective functions that we then combine into one.

    Informally, let $\alpha'$ be $A$, and let $\beta'$ be $r(A)$. Define $r'$ as the relation $r$ restricted to $\alpha'$ and $\beta'$. Formally, this looks like
\begin{lstlisting}
let α' := {a' : α // a' ∈ A},              -- the type of elements of A
let β' := {b' : β // b' ∈ image_rel r A},  -- the type of elements of r(A)
let r' : α' → β' → Prop := λ a' b', r a' b',
\end{lstlisting}
    Since $|A| = |r(A)|$, we have that for every subset $S \subseteq A$, $|S| \leq |r'(S)|$. This means that for all finite sets $A'$ defined on $\alpha'$, $|A'| \leq |r'(A')|$, and we can apply the induction hypothesis to get injective function $f' : \alpha' \rightarrow \beta'$ that respects relation $r'$.

    Now, informally, let $\alpha''$ be $\alpha$ with $A$ removed, and let $\beta''$ be $\beta$ with $r(A)$ removed. Define $r''$ as the relation $r$ restricted to $\alpha''$ and $\beta''$. Formally, we have
\begin{lstlisting}
let α'' := {a'' : α // a'' ∉ A},             -- the type of elements not in A
let β'' := {b'' : β // b'' ∉ image_rel r A}, -- the type of elements not in r(A)
let r'' : α'' → β'' → Prop := λ a'' b'', r a'' b'',
\end{lstlisting}
    If $A''$ is a finite set defined on $\alpha''$, we have that $A''$ and $A$ are disjoint in $\alpha$, so $r(A \cup A'') = r'(A) \cup r''(A'')$, and $|r(A \cup A'')| = |r'(A)| + |r''(A'')|$. Since we assumed that $|A \cup A''| \leq |r(A \cup A'')|$, and $|A| = |r'(A)|$, we get that $|A| + |A''| = |A \cup A''| \leq |r(A \cup A'')| = |r'(A)| + |r''(A'')|$, or $|A''| \leq |r''(A'')|$. This means that Hall's condition is satisfied for $\alpha''$, and we can apply our induction hypothesis to get injective function $f'' : \alpha'' \rightarrow \beta''$ that respects relation $r''$.

    Since $\alpha'$ and $\alpha''$ are disjoint, as are $\beta'$ and $\beta''$, we have that the images of $f'$ and $f''$ are also disjoint. So, to get our injective function $f : \alpha \rightarrow \beta$, we combine $f'$ and $f''$.
\end{proof}

These two cases are combined in \lean{hall_hard_inductive_step}, completing the proof of \lean{hall}.

\subsection{Via graph theory}

We give an overview of the simple graph library of \mathlib{} that implements the graph theory definitions from \Cref{sec:hall-via-graphs}, and then in the following section we use our formalization of the theorem via relations to prove the graph theoretical version.

\subsubsection{Graph theory library}

The simple graph library of \mathlib{} (in \lean{combinatorics.simple_graph}) was initially created by Aaron Anderson, Jalex Stark, and the third author in Summer of 2020, and its first application was to support the formalization of the Friendship Theorem to add to Lean's standing in Freek Wiedijk's 100 theorems list\cite{FreekList}.
There is a larger ambition to design a cohesive library that includes directed graphs, multigraphs, and combinatorial maps, too, but it was thought that simple graphs were a good testbed for the design of a combinatorics library, since they already present interesting design issues that algebraic objects tend not to have.
In \cite{Chou1994}, Chou, who gave a formalization of multigraphs in HOL, mentioned the difficulty in formalizing such concrete objects --- de Bruijn had pointed out that the more abstract a piece of mathematics, the more gaps that mathematicians had already filled, making the process of formalization easier.

The type of simple graphs is based on the one in \cite{Doczkal2019} for Coq:
\begin{lstlisting}
structure simple_graph (V : Type u) :=
(adj : V → V → Prop)
(sym : symmetric adj)
(loopless : irreflexive adj)
\end{lstlisting}
Given a type \lstinline{V}, then \lstinline{G : simple_graph V} is a simple graph defined over \lstinline{V} with an adjacency relation \lstinline{G.adj} between vertices of type \lstinline{V} that is symmetric (\lstinline{G.sym}) and irreflexive (\lstinline{G.loopless}).
One does not generally speak of a type of \emph{all} graphs, but if necessary the sigma type \lean{Σ (V : Type u), simple_graph V} gives all the simple graphs in universe \lean{Type u}.
A reason for this is that if the vertex type were included in the structure, then equalities of simple graphs would involve equalities of their vertex types, and \emph{disproving} equalities of types tends to be impossible --- there are models of Lean's formal system where types of the same cardinality are equal --- and applications where vertices of different graphs would need to be compared tend to be better served by, for example, situating the graphs as subgraphs in some larger graph.

As an example of new design challenges for graphs that have not already been solved in \mathlib{}, consider the following.
A difference between simple graphs and other algebraic objects such as groups in common mathematical practice is that, in the statement ``$G$ is a group'', $G$ refers to the set of elements of the group, with the group operation and proofs of group axioms being left tacit.
In contrast, in ``$G$ is a simple graph,'' $G$ refers to the simple graph itself.
This is all to say that algebraic objects experience \emph{synecdoche} --- a part refers to the whole.
Algebraic objects in \mathlib{} are similarly referred to by their carrier types, with the synecdoche being implemented through typeclasses.
For objects that are referred to by a component type or function,
Lean's automatic coercion system helps make subobjects ``be'' objects themselves.
For example, subgroups of a group \lean{G} are a structure consisting of a carrier set \lean{G' : set G} and proofs of closure, and there is a coercion from subgroups to types defined by coercing \lean{G'} to be a type.
There is a typeclass instance giving this type the structure of a group, and in this sense a subgroup ``is'' a group.
These systems do not apply to simple graphs, which are neither types nor functions, and Lean's other automatic coercion system is currently unable to let subgraphs ``be'' simple graphs without manual inclusion of universe levels that it is unable to infer itself.

Simple graphs provide the following functions to give the neighbor set of a single vertex, the neighbor set of a set of vertices, and the set of edges.
We expand some of the definitions here from how they appear in \mathlib{ }to reduce the number of dependencies in this overview.
\begin{lstlisting}
/-- The set of all `w` adjacent to a given `v`. -/
def neighbor_set (v : V) : set V := {w : V | G.adj v w}

/-- The set of all `w` adjacent to an element of `S`. -/
def neighbor_set_image (S : set V) : set V :=
{w : V | ∃ v, v ∈ S ∧ w ∈ G.neighbor_set v}

/-- The set of all unordered pairs `⟦(v, w)⟧` such that `G.adj v w` -/
def edge_set : set (sym2 V) := sym2.from_rel G.sym
\end{lstlisting}
The type \lean{sym2 V} is the \emph{symmetric square} of \lean{V}, which consists of all unordered pairs, represented as $V\times V$ modulo permutations.
The \lean{edge_set} is the set of all unordered pairs $[(v,w)]$ such that $v$ is adjacent to $w$.
Formally, if $q:V \times V \to \operatorname{Sym}^2V$ is the quotient map, then \lean{sym2.from_rel h} is defined to be the dashed arrow in the following commutative diagram when \lean{h} is a proof that $r:V\to V\to\text{Prop}$ is a symmetric relation.
\begin{equation*}
    \begin{tikzcd}
        V \times V \arrow[rr,"q"] \drar[swap]{\operatorname{uncurry} r} & & \operatorname{Sym}^2 V \dlar[dashed]{} \\
        & \text{Prop}
    \end{tikzcd}
\end{equation*}
The type $\operatorname{set}(\operatorname{Sym}^2V)$ is defined to be $\operatorname{Sym}^2 V \to \text{Prop}$.

While it is not necessary for this paper, we will note that Lean's typeclass system makes it convenient to work with degrees of infinite graphs:
\begin{lstlisting}
def degree (v : V) [fintype (G.neighbor_set v)] : ℕ :=
(G.neighbor_set v).to_finset.card
\end{lstlisting}
The square brackets indicate that Lean will try to fill in the second argument automatically with a proof of finiteness of the neighbor set of \lean{v} in the expression \lean{G.degree v}.
If \lean{V} is a finite type, then there is an instance that is able to automatically supply such a proof.

For graph colorings, we use a ``bundled'' type, combining the coloring function $V\to C$ with a proof that it is a proper coloring.
\begin{lstlisting}
/-- `G.coloring C` is the type of `C`-colorings of `G`. -/
structure coloring (G : simple_graph V) (C : Type v) :=
(color : V → C)
-- Adjacent vertices have distinct colors:
(valid : ∀ ⦃v w : V⦄, G.adj v w → color v ≠ color w)

/-- The set of vertices in the color class for `c`. -/
def coloring.color_set (c : C) : set V := f.color ⁻¹' {c}

/-- A bipartition `f : G.bipartition` is a coloring of `G` by
    the two-term type `fin 2`.  The color classes `f.color_set 0`
    and `f.color_set 1` give the partition of `V`. -/
def bipartition (G : simple_graph V) := G.coloring (fin 2)
\end{lstlisting}

We also define a ``bundled'' type for subsets of edges that form a matching of a given graph.
\begin{lstlisting}
structure matching (G : simple_graph V) :=
(edges : set (sym2 V))
(sub_edges : edges ⊆ G.edge_set)
-- If two edges are in the matching, and if v is a vertex incident to both,
-- then the edges are the same:
(disjoint : ∀ (x y ∈ edges) (v : V), v ∈ x → v ∈ y → x = y)

def matching.saturates (M : G.matching) (S : set V) : Prop :=
S ⊆ {v : V | ∃ x, x ∈ M.edges ∧ v ∈ x}
\end{lstlisting}
If \lstinline{M : G.matching}, then \lstinline{M} is made up of an edge set \lstinline{M.edges} that is a subset of the edge set of \lstinline{G} (\lstinline{M.sub_edges}) where no two edges share an incident vertex (\lstinline{M.disjoint}). If \lstinline{M.saturates S} where \lstinline{S} is a set of vertices from \lstinline{V}, then for all \lstinline{v ∈ S}, there exists some edge \lstinline{x ∈ M.edges} incident to \lstinline{v}.

\subsubsection{Formalized theorem}
\label{sec:hall-theorem-formal}

In the context of the \mathlib{} graph theory library, the statement of \Cref{thm:hall-graph} is as follows, simplified to the case of a finite bipartite graph:
\begin{lstlisting}
variables (G : simple_graph V) [fintype V] (b : G.bipartition)

theorem hall_marriage_theorem :
  (∀ (S ⊆ (b.color_set 0)),
     fintype.card S ≤ fintype.card (G.neighbor_set_image S))
  ↔ (∃ (M : G.matching), M.saturates (b.color_set 0))
\end{lstlisting}
Modulo technical details of supplying finiteness proofs to \lstinline{fintype.card}, the full statement of \Cref{thm:hall-graph} is given by
\begin{lstlisting}
variables (G : simple_graph V) (b : G.bipartition)
variables [fintype (b.color_set 0)] 
variables [∀ v, v ∈ b.color_set 0 → fintype (G.neighbor_set v)]

theorem hall_marriage_theorem' :
  (∀ (S ⊆ (b.color_set 0)),
     fintype.card S ≤ fintype.card (G.neighbor_set_image S))
  ↔ (∃ (M : G.matching), M.saturates (b.color_set 0))
\end{lstlisting}

We will now describe how to prove \lean{hall_marriage_theorem} from \lean{hall} in \Cref{sec:formal-via-relations}.
Recall its statement:
\begin{lstlisting}
variables {α β : Type u} [fintype α] [fintype β]
variables (r : α → β → Prop)

theorem hall :
  (∀ (A : finset α), A.card ≤ (image_rel r A).card)
    ↔ (∃ (f : α → β), function.injective f ∧ ∀ x, r x (f x))
\end{lstlisting}
We specialize \lstinline{hall} to match up with the definitions we have. In our case, we have that \lstinline{α} and \lstinline{β} are \lstinline{b.color_set 0} and \lstinline{b.color_set 1}, respectively. Our relation \lstinline{r} is simply \lstinline{G.adj} when restricted to \lstinline{b.color_set 0} and \lstinline{b.color_set 1}. A matching \lean{f} may be regarded as a matching in the graph that saturates \lstinline{b.color_set 0}, and conversely a matching that saturates \lstinline{color_set 0} defines a matching of \lean{r} by \lstinline{matching.disjoint} and \lstinline{matching.sub_edges}.
We also have that \lstinline{image_rel r A} is equivalent to \lstinline{G.neighbor_set A}.

\subsection{K\H{o}nig's lemma}
\label{sec:konig-lemma-formal}

The usual K\H{o}nig's lemma is that every countably infinite locally finite simple graph $G$ has an infinite ray, which is a sequence of infinitely many distinct vertices $v_1,v_2,\dots$ such that $v_i$ is adjacent to $v_{i+1}$ for all $i$.

For our purposes, we use a different formulation.
Consider an $\mathbb{N}$-indexed inverse system of sets 
\begin{equation*}
    \begin{tikzcd}
    X_0 & X_1 \lar[swap]{f_0} & X_2 \lar[swap]{f_1} & X_3 \lar[swap]{f_2} & \lar[swap]{f_3} \cdots
    \end{tikzcd}.
\end{equation*}
The \emph{inverse limit} (or \emph{projective limit}) $\varprojlim_{i}X_i$ is the limit of the above diagram.
Elements of the inverse limit are elements $x\in \prod_{i}X_i$ such that $x_i=f_i(x_{i+1})$ for all $i\in\mathbb{N}$.
For us, K\H{o}nig's lemma is the statement that if each $X_i$ is a nonempty finite set then the inverse limit is nonempty:
\begin{lemma}[K\H{o}nig's lemma]
    Suppose $\{X_i\}_{i\in\mathbb{N}}$ is an indexed family of sets with functions $f_i:X_{i+1}\to X_i$ for each $i$.
    If each $X_i$ is a nonempty finite set, then there exists a family of elements $x\in\prod_iX_i$ such that $x_i=f_i(x_{i+1})$ for all $i$.
\end{lemma}
The usual K\H{o}nig's lemma follows from this one by selecting a vertex $v_0$ then letting $X_i$ consist of all paths from $v_0$ of length at most $i$, with $f_i$ being path truncation --- the set of infinite rays corresponds to the inverse limit of this inverse system.

\begin{remark}
    Inverse limits are not always nonempty, even if every $X_i$ is nonempty.
    For example, consider the inverse system given by $X_i=\{j\in\mathbb{N}\mid j\geq i\}$ with inclusion maps $f_i:X_{i+1}\hookrightarrow X_{i}$ for each $i$.
\end{remark}

A similar version of K\H{o}nig's lemma, in terms of infinite rays in trees, was formalized in the Mizar proof assistant in \cite{Bancerek1991}.
We have formalized a proof of K\H{o}nig's lemma in Lean, but we still need to adapt it to fit into the pre-existing \mathlib{} category theory library.
The Lean code for the entire proof is about 200 lines using only basic definitions of finite sets.
The first ingredient is inverse systems.
We only consider indexed families of finite sets, and since it is possible to create a universal type that these are subsets of, we define inverse systems like so:
\begin{lstlisting}
structure inv_system (α : Type u) :=
(ι : ℕ → finset α)
(f : α → α)
(fprop : ∀ (n : ℕ) (x : α), x ∈ ι n.succ → f x ∈ ι n)
\end{lstlisting}
For simplicity of dependent types, the function \lstinline{f} is given as an endomorphism of this universal type rather than as
\begin{lstlisting}
structure inv_system' :=
(ι : ℕ → Type v)
(f' : Π (n : ℕ), ι n.succ → ι n)
\end{lstlisting}
whose direct use has difficulties especially when trying to rewrite types of terms such as \lean{x : ι ((n + 1) + k)} to \lean{x : ι (n + (k + 1))}, which are not definitionally equal.

Then, the type of all limits of an inverse system is given by
\begin{lstlisting}
structure inv_system.limit {α : Type u} (ι : inv_system α) :=
(s : ℕ → α)
(s_mem : ∀ (n : ℕ), s n ∈ ι n)
(is_lim : ∀ (n : ℕ), ι.f (s n.succ) = s n)
\end{lstlisting}
The statement of our version of K\H{o}nig's lemma is thus
\begin{lstlisting}
lemma limit_nonempty_if_finite_and_nonempty
  {α : Type u} (ι : inv_system α)
  (hne : ∀ n, (ι n).nonempty) :
  nonempty ι.limit
\end{lstlisting}
As the lemma is nonconstructive, the lemma produces a proof of \lstinline{nonempty ι.limit} rather than a term of \lstinline{ι.limit}.

The key to proving K\H{o}nig's lemma is to restrict to \emph{very extendable} elements of the inverse system, which are elements $x\in X_n$ such that for every $k$ there is a $y\in X_{n+k}$ such that $x=f_nf_{n+1}\dots f_{n+k-1}y$.
The restriction is also an inverse system, and since in the restriction every map is surjective, elements of the inverse limit can be produced by repeatedly choosing lifts of a very extendable element of $X_0$.
The core of the proof is to use the fact that $X_i$ is finite and nonempty to show that the restricted inverse system is also finite and nonempty.
If there were an $X_n$ with no very extendable elements, then one could take the maximum $k$ over the elements of $X_n$ such that a $y$ in the definition of very extendable does not exist.  Since $X_{n+k}$ is nonempty, the iterated image of an element of $X_{n+k}$ in $X_n$ gives an element contradicting the maximality of $k$.

Recall the statement of the infinite version of the Hall Marriage Theorem from \Cref{thm:infinite-hall-indexed-family}, specialized to a countably infinite index set.
Formalized in Lean, this is
\begin{lstlisting}
theorem infinite_hall {α : Type u} {β : Type v} (ι : α → finset β) (h : ℕ ≃ α) :
  (∀ (s : finset α), s.card ≤ (s.bind ι).card) ↔ nonempty (matching ι)
\end{lstlisting}
The type \lean{ℕ ≃ α} of equivalences consists of inverse pairs of functions between \lean{ℕ} and \lean{α}.  These coerce to functions \lean{ℕ → α}, and the inverse equivalence is given by \lean{h.symm : α ≃ ℕ}.

The way in which K\H{o}nig's lemma can be used to prove this is to let $M_n$ denote the set of all matchings on the first $n$ indices of $\alpha$, with $f_n:M_{n+1}\to M_n$ being the restriction of a matching to a smaller index set.
The Hall marriage condition applies to the restricted indexed families, hence proves each $M_n$ is nonempty, and so the conditions for K\H{o}nig's lemma are satisfied.
An element of the inverse limit gives a way to produce a matching on the entire index set: for an index within the first $n$ indices, consider the inverse limit's matching in $M_{n}$ and use this matching to give the value associated to the element.  This is independent of the choice of $n$, so well-defined and injective.

Since we did not use the proper \mathlib{} interface, which would be akin to \lean{inv_system'}, we ran into some dependent type difficulties in the proof, but all were eventually solvable by constructing technical lemmas for difficult rewrites.
The inverse system was defined over \lean{Σ (n : ℕ), matching (fin_restrict ι h n)}, where \lean{fin_restrict ι h n} restricts an indexed family of sets to the first \lean{n} indices, where \lean{(ι k).fst = k} and \lean{(ι k).snd} is a matching on the first \lean{k} indices, represented as terms of \lean{fin k}.
The main difficulty was in manipulating expressions of the form
\begin{lstlisting}
(L.s (h.symm a).succ).snd ⟨h.symm a, _⟩
\end{lstlisting}
where \lean{L} is a limit of the inverse system of matchings.

\subsection{An application}
\label{sec:application}

Say a graph $G$ with vertex set $V$ \emph{carries} a function $f: V\to V$ if (1) $v$ is adjacent to $f(v)$ for all $v\in V$ and (2) $f(f(v))\neq v$ for all $v$ (that is, if the edges $[(v, f(v))]$ and $[(w, f(w)]$ are equal, then $v=w$).
We give the condition under which a finite graph carries a function and give a Lean formalization of the statement.

\begin{theorem}
    \label{thm:carried-fun-iff}
    A graph $G$ carries a function if and only if for every subset $U\subseteq V$ that the total number of edges incident to a vertex in $U$ is at least $\lvert U\rvert$.
\end{theorem}
\begin{proof}
    For each $v\in V$, let $E_v$ be the set of edges incident to $v$.
    A function $f$ carried by $G$ can be thought of as being a choice of edge in $E_v$ for each $v$ such that different vertices choose different edges --- the chosen edge records the next image of $f$.
    Hence, the theorem is equivalent to \Cref{thm:infinite-hall-indexed-family}.
\end{proof}

We now give a formalization of this statement, specialized to finite graphs, along with the key lemma in its proof using the formalization from \Cref{sec:formalize-indexed-families}.
Consider the following variables.
\begin{lstlisting}
variables {V : Type u} [fintype V]
\end{lstlisting}
The statement of \Cref{thm:carried-fun-iff} is then
\begin{lstlisting}
theorem exists_carried_fun_iff {V : Type u} [fintype V]
  (G : simple_graph V) :
  (∃ (f : V → V), (∀ v, G.adj v (f v))
                  ∧ function.injective (λ v, ⟦(v, f v)⟧))
  ↔ (∀ (U : finset V), U.card ≤ (U.bind (λ u, G.incidence_finset u)).card)
\end{lstlisting}
The core of the proof is
\begin{lstlisting}
lemma exists_carried_iff_matching (G : simple_graph V) :
  (∃ (f : V → V), (∀ v, G.adj v (f v))
                  ∧ function.injective (λ v, ⟦(v, f v)⟧))
  ↔ nonempty (matching (λ u, G.incidence_finset u))
\end{lstlisting}
and it is a standard application of lemmas in the simple graph library.
Rewriting the goal using this lemma results in a goal that is immediately proved by \lean{hall}.

\begin{remark}
    There is also a version of this theorem for directed graphs, which have not yet been formalized in \mathlib{}.
    We say a directed graph $G$ \emph{carries a function} $f:V\to V$ if $(v,f(v)$ is a directed edge for every $v\in V$.
    A finite directed graph carries a function if and only if for every subset $U\subseteq V$ that the total number of directed edges $(u,v)$ in $G$ with $u\in U$ is at least $\lvert U\rvert$.
\end{remark}

\section{Future Work}
\label{sec:future}

Since Hall's Marriage Theorem is usually presented along with graph theory in an undergraduate mathematics course (indeed, it appears shortly after the degree-sum formula in \cite{Rosen2011}), we had naively assumed this formalization project would help develop more of the \mathlib{} graph theory library --- in that vein, a way in which the \mathlib{} graph theory community is evaluating progress is to periodically check how many pages of \cite{Bollobas1998} have been formalized.

For design purposes, it is useful to have specific targets in mind, and so we mention a few projects that we think would be worth pursuing next.

\begin{itemize}
\item A matching $M$ of a graph $G$ is \emph{perfect} if $M$ saturates the vertex set $V$.
    Tutte's theorem is that a graph has a perfect matching if and only if for every $U\subseteq V$ then the induced graph $G[V-U]$ has at most $\lvert U\rvert$ connected components with an odd number of vertices.
    
    While there is code for induced subgraphs and connected components, it has not yet been contributed to \mathlib{} proper.
    This theorem would help refine the design of these to make sure they interoperate well.
    
\item A \emph{weighted} graph is a graph along with an assignment of numbers, called \emph{weights}, to the edges.  The weight of a matching is the sum of the weights of its edges.  The \emph{assignment problem} is a combinatorial optimization problem that involves finding maximal (or minimal) matchings with respect to weight.  The Hungarian Algorithm solves computes maximal perfect matchings.  It would be interesting to formalize the algorithm, to prove it runs in polynomial time, and to show it can find maximal perfect matchings.  This could also be adapted to show it can find maximal weight transversals for matrices.

\item For the category theory part of \mathlib{}, it would be worth formalizing the full infinite cardinality version of Hall's theorem (\Cref{thm:infinite-hall-indexed-family}).  Inverse systems already exist, and this would involve adding the theorem that inverse systems of nonempty finite types have a nonempty inverse limit.  In more generality, we could prove that inverse systems of compact spaces in the category of topological spaces have nonempty inverse limits.

\end{itemize}

\section*{Acknowledgments}

We would like to thank Kevin Buzzard for his suggestion to work on this project and for his constant encouragement. We also would like to thank Jalex Stark for helpful comments on a draft, as well as for contributing code to the \mathlib{} repository associated to this paper.

\section*{Declaration of interest}

No potential conflict of interest was reported by the author(s). 

\section*{Funding}

The second author was supported by the Cantab Capital Institute for the Mathematics of Information.
The third author was supported by the Simons Foundation.

\let\MRhref\undefined
\bibliographystyle{hamsalpha}
\bibliography{sources.bib}

\end{document}